\documentclass[reqno,11pt]{amsart}
\usepackage[letterpaper,margin=1in,footskip=0.25in]{geometry}
\usepackage[utf8]{inputenc}
\usepackage{mathptmx}
\usepackage[cal=boondoxupr]{mathalfa}
\usepackage{mathrsfs}
\usepackage{setspace, amssymb, amsopn, amsmath, array, pdfsync, url, amsfonts, float,enumitem}
\usepackage{xcolor}
\definecolor{chianti}{rgb}{0.6,0,0}
\definecolor{meretale}{rgb}{0,0,.6}
\definecolor{leaf}{rgb}{0,.35,0}

\usepackage[colorlinks=true, hyperindex, citecolor=meretale, urlcolor=leaf, linkcolor=chianti]{hyperref}
\usepackage{tikz, tikz-cd}
\usepackage{fancyvrb,newverbs}
\usepackage[capitalise]{cleveref}
\usepackage{multicol}
\usepackage{booktabs}

\definecolor{cverbbg}{gray}{0.93}

\allowdisplaybreaks

\newtheorem{Theoremx}{Theorem}
 
\newtheorem{theorem}{Theorem}[section]
\theoremstyle{definition}

\theoremstyle{definition}

\newtheorem{lemma}[theorem]{Lemma}

\newtheorem{proposition}[theorem]{Proposition}
\newtheorem{corollary}[theorem]{Corollary}

\newtheorem{conjecture}[theorem]{Conjecture}
\theoremstyle{definition}
\newtheorem{definition}[theorem]{Definition}
\newtheorem{remark}[theorem]{Remark}
\theoremstyle{remark}

\renewcommand{\ker}{\operatorname{ker}}

\newcommand{\Spec}{\operatorname{Spec}}
\newcommand{\ts}{\textsuperscript}

\usepackage{stmaryrd}

\newcommand{\Cl}{\operatorname{Cl}}

\newcommand{\Soc}{\operatorname{Soc}}
\newcommand{\Hom}{\operatorname{Hom}}

\newcommand{\Char}{\operatorname{char}}

\newcommand{\Frac}{\operatorname{Frac}}

\newcommand{\Z}{\mathbb{Z}}
\newcommand{\F}{\mathbb{F}}

\newcommand{\fm}{\mathfrak{m}}
\newcommand{\fp}{\mathfrak{p}}
\newcommand{\fq}{\mathfrak{q}}
\newcommand{\fa}{\mathfrak{a}}
\newcommand{\fb}{\mathfrak{b}}
\newcommand{\fc}{\mathfrak{c}}

\newcommand{\fn}{\mathfrak{n}}

\newcommand{\fQ}{\mathfrak{Q}}

\newcommand{\rperf}{R_{\operatorname{perf}}}
\newcommand{\rperfd}{R_{\operatorname{perfd}}}

\newcommand{\przero}{\operatorname{pr}_0}

\makeatother

\crefname{Theoremx}{Theorem}{Theorems}
\crefname{setting}{Setting}{Settings}

\makeatletter
\renewcommand*{\eqref}[1]{%
  \hyperref[{#1}]{\textup{\tagform@{\ref*{#1}}}}%
}
\makeatother

\usepackage[backend=biber, style=numeric, datamodel=mrnumber,sorting=anyt,maxalphanames=5,maxbibnames=99]{biblatex}
\usepackage{hyperref}
\usepackage{xurl}
\hypersetup{breaklinks=true}

\DeclareFieldFormat{mrnumber}{%
  MR\addcolon\space
  \ifhyperref
    {\href{http://www.ams.org/mathscinet-getitem?mr=#1}{\nolinkurl{#1}}}
    {\nolinkurl{#1}}}

\renewbibmacro*{doi+eprint+url}{%
  \iftoggle{bbx:doi}
    {\printfield{doi}}
    {}%
  \newunit\newblock
  \printfield{mrnumber}%
  \newunit\newblock
  \iftoggle{bbx:eprint}
    {\usebibmacro{eprint}}
    {}%
  \newunit\newblock
  \iftoggle{bbx:url}
    {\usebibmacro{url+urldate}}
    {}}

\renewbibmacro{in:}{}

\begin{document}

\title{On deformation of perfectoid purity in Gorenstein domains}

\author[Baily]{Benjamin Baily}
\thanks{Baily was supported by NSF RTG grant DMS \#1840234}
\address{Department of Mathematics, University of Michigan, Ann Arbor, MI 48109 USA}
\email{bbaily@umich.edu}
\urladdr{\url{https://bbaily.github.io}}

\author[Dovgodko]{Karina Dovgodko}
\address{Department of Mathematics, Columbia University, New York, NY 10027 USA}
\email{kmd2235@columbia.edu}

\author[Simpson]{Austyn Simpson}
\thanks{Simpson was supported by NSF postdoctoral fellowship DMS \#2202890.}
\address{Department of Mathematics, University of Michigan, Ann Arbor, MI 48109 USA}
\email{austyn@umich.edu}
\urladdr{\url{https://austynsimpson.github.io}}

\author[Westbrook]{Jack Westbrook}
\address{Department of Mathematics, University of Wisconsin, Madison, WI 53706}
\email{jswestbrook@wisc.edu}

\begin{abstract}
    If $(R,\fm)$ is a complete local ring of mixed characteristic $(0,p)$ and $R/pR$ is an $F$-pure Gorenstein domain, we find a sufficient condition for $R$ to be perfectoid pure. This condition is related to the Cohen--Macaulayness of the absolute integral closures of Gorenstein local domains of mixed characteristic which are not necessarily excellent. Along the way, we show that the problem of lifting $F$-purity of $R/pR$ to perfectoid purity of $R$ is equivalent to a similar deformation problem for the splinter property.
\end{abstract}

\maketitle

\section{Introduction}
Perfectoid purity was recently proposed in \cite{BMPSTWW24} as a mixed characteristic analogue of $F$-purity in prime characteristic $p>0$ and log canonicity in equal characteristic zero. A noetherian ring $R$ whose Jacobson radical contains the prime integer $p$ is said to be \emph{perfectoid pure} if there exists a pure map $R\to B$ where $B$ is a perfectoid $R$-algebra. It is shown in \emph{op. cit.} that this notion enjoys many of the properties that a counterpart to $F$-purity ought to, such as weak normality \cite[Corollary 4.29]{BMPSTWW24} and ascent along cyclic covers \cite[Proposition 5.8]{BMPSTWW24}.

The goal of this article is to investigate conditions under which the property of being perfectoid pure \emph{deforms}; that is, whether a local ring $(R,\fm)$ with $p\in\fm$ is perfectoid pure provided that $R/xR$ is perfectoid pure for some nonzero divisor $x\in\fm$. This is settled already in some cases such as \cite[Theorem 6.6]{BMPSTWW24} and \cite[Theorem A]{Yos25} which both require the ring $R$ to be a complete intersection. An affirmative answer to the deformation problem is expected merely assuming that $R$ is Gorenstein in light of the analogous fact for $F$-purity \cite{Fed83}. We consider in this article an instance of this conjecture when the nonzero divisor is the prime integer $p$ itself.
\begin{conjecture}\label{conj:F-pure}
    Let $(R,\fm, k)$ be a complete local Gorenstein ring of mixed characteristic $(0,p)$. If $R/pR$ is an $F$-pure domain, then $R$ is perfectoid pure.
\end{conjecture}

Much like many deformation arguments for local singularity types, the proof in the complete intersection setting leverages that a certain object---in this case the perfectoidization $\rperfd$---is Cohen--Macaulay in the sense that $H^i_\fm(\rperfd)=0$ for all $i<\dim(R)$ \cite[Theorem 4.24]{BMPSTWW24}. We find a relationship between \cref{conj:F-pure} and the vanishing of local cohomology of a different algebra, namely that of the absolute integral closure $R^+$. Note that $\rperfd$ in this context is a derived object whereas $R^+$ is a discrete ring.

\begin{conjecture}\label{conj:CM}
    Let $(R,\fm,k)$ be a $d$-dimensional noetherian (not necessarily excellent) local Gorenstein domain of mixed characteristic $(0,p)$. Then $H^i_\fm(R^+)=0$ for all $i<d$.
\end{conjecture}

Our main theorem is the following relationship between these conjectures.

\begin{Theoremx} (= \cref{thm:perfd-pure-deformation})\label{thm:main-theorem-1}
    \cref{conj:CM} implies \cref{conj:F-pure}.
\end{Theoremx}

Some remarks are in order regarding the evidence for \cref{conj:CM}. The conclusion is never satisfied for local rings of dimension at least three containing the rationals, see \cite[page 55]{HH92}. However, this vanishing is enjoyed by $R^+$ for excellent local domains $R$ containing $\F_p$ \cite[Theorem 1.1]{HH92}. This result was subsequently extended to include all local domains which are homomorphic images of Gorenstein (resp. Cohen--Macaulay) rings in \cite[Corollary 2.3]{HL07} (resp. \cite[Corollary 3.2]{Quy16}). Currently the best progress in mixed characteristic is Bhatt's result \cite[Theorem 1.1]{Bha21} which states that $H^i_\fm(R^+)=0$ for all $i<\dim R$ when $R$ is an excellent local domain of mixed characteristic $(0,p)$. It thus seems plausible that \emph{all} Gorenstein local domains of mixed characteristic have this property, and indeed the question of whether this holds in the level of generality of either \cite{HL07} or \cite{Quy16} is posed in \cite[Question 5.3]{Bha21}.

We emphasize that \cref{conj:F-pure} concerns \emph{complete} local rings, while \cref{conj:CM} only remains unresolved outside of the excellent scenario. This points to some surprising features of our proof that we briefly summarize. We show that if $R/pR$ is a Gorenstein $F$-pure local domain, then $R\llbracket x\rrbracket$ contains a local subring $A$ such that:
\begin{enumerate}[label=(\alph*)]
    \item $\widehat{A} \cong R\llbracket x\rrbracket$,
    \item $pR\cap A = pA$, and 
    \item $A/pA$ is \emph{weakly $F$-regular}.
\end{enumerate}

Weakly $F$-regular rings (i.e., rings in which all ideals are tightly closed) satisfy the splinter condition, a characteristic-independent notion prescribing that any map to a finite cover splits. Granting \cref{conj:CM}, this local Gorenstein (non-excellent) ring $A$ will also be a splinter (\cref{lem:splinter lifts to splinter}) and hence perfectoid pure (\cref{thm:splinter-implies-perfd-pure}). Using results from \cite{BMPSTWW24}, perfectoid purity can then be ascended along the completion $A\to R\llbracket x\rrbracket$ and then descended along $R\to R\llbracket x\rrbracket$.

Our proof of \cref{thm:main-theorem-1} also shows that, irrespective of \cref{conj:CM}, the problem of deforming $F$-purity of $T/pT$ to perfectoid purity of the complete ring $T$ is equivalent to deforming weak $F$-regularity to perfectoid purity in Gorenstein domains which are not assumed to be excellent:

\begin{conjecture}\label{conj:WFR}
    Let $(R,\fm)$ be a local Gorenstein ring of mixed characteristic $(0,p)$. If $R/pR$ is weakly $F$-regular and analytically irreducible, then $R$ is perfectoid pure.
\end{conjecture}
We expect an even stronger conclusion to hold than what is stated above, namely that $R$ is a splinter; indeed, this is implied by \cref{conj:CM} (see \cref{lem:splinter lifts to splinter}). Nevertheless, we have:

\begin{Theoremx} (= \cref{thm:conj1 iff conj2})\label{thm:main-theorem-2}
\cref{conj:F-pure} is equivalent to \cref{conj:WFR}.
\end{Theoremx}

The proof strategy outlined above adds to a long-running program pioneered by Lech \cite{Lec86} and Heitmann \cite{Hei93} aiming to show that a complete local ring which lacks some mildness property---in this case the splinter condition---arises as the completion of a ring \emph{with} that property.

Finally, we caution the reader about relaxing the assumptions of \cite[Theorem 6.6]{BMPSTWW24} or those of \cref{conj:F-pure} much further beyond the Gorenstein case. Indeed, there are examples of Fedder \cite{Fed83} and Singh \cite{Sin99b,Sin99a} demonstrating that $F$-purity does not deform, although there are positive deformation results provided that the canonical module $\omega_R$ is torsion in the class group $\Cl(R)$ \cite{PS23,PST25}. Note also that the analytic irreducibility assumption in \cref{conj:WFR} is irredundant, as there are examples of Gorenstein (weakly) $F$-regular domains whose completions fail to be domains by \cite[Theorem C]{LS25}.

\subsection*{Some conventions}
By a \emph{quasi-local} ring $(R,\fm)$, we mean a commutative unital ring with unique maximal ideal $\fm$. A \emph{local ring} is a quasi-local ring which is also noetherian. For a local ring $(R,\fm)$, we denote by $R^\circ$ the elements of $R$ which are not contained in any minimal prime. The \emph{punctured spectrum} of $R$, denoted $\Spec^\circ(R)$, refers to the nonmaximal prime ideals of $R$. If $R$ is an $\F_p$-algebra, then $F^e:R\to R$ denotes the $e$\ts{th} iterated Frobenius endomorphism $r\to r^{p^e}$.

\subsection*{Acknowledgments} We are grateful to Susan Loepp and Karl Schwede for helpful discussions, and we thank Ryo Ishizuka and Thomas Polstra for comments on an earlier draft. We also thank Linquan Ma and Karl Schwede for pointing out an inaccuracy in a previous draft.
\vspace{-.2cm}
\section{Preliminaries}
We review the various singularity types that we will consider; these include $F$-purity and (weak) $F$-regularity in prime characteristic, the splinter condition in arbitrary characteristic, and perfectoid purity in mixed characteristic. Recall that if $R$ is a commutative ring then a map $N_1\to N_2$ of $R$-modules is \emph{pure} if the induced map $M\otimes_R N_1\to M\otimes_R N_2$ is injective for every $R$-module $M$. The map $N_1\to N_2$ is said to be \emph{cyclically pure} if for every ideal $I\subseteq R$, the map $N_1/IN_1\to N_2/IN_2$ is injective. Pure maps are clearly always cyclically pure and the two notions coincide in all settings of interest in this article \cite{Hoc77}.

Let $p>0$ be a prime integer. A reduced noetherian $\F_p$-algebra $R$ is \emph{$F$-pure} if the inclusion $R\hookrightarrow R^{1/p}$ is pure, i.e. if for any $R$-module $M$, the induced map $M\to M\otimes_R R^{1/p}$ is injective. It is straightforward to see that if $R$ is $F$-pure and $I\subseteq R$ is an ideal then $I =I^F$, where the \emph{Frobenius closure} $I^F$ of $I$ is the ideal consisting of elements $r\in R$ such that $r^{q} \in F^e(I)R$ for all sufficiently large $q =p^e$. If $I\subseteq R$ is an ideal, then the \emph{tight closure of $I$}, denoted $I^*$, is the ideal consisting of elements $r\in R$ for which there exists $c\in R^\circ$ such that $cr^q\in F^e(I)R$ for all sufficiently large $q=p^e$. We say that $R$ is \emph{weakly $F$-regular} if $I = I^*$ for all ideals $I\subseteq R$. If $R$ is a domain, the containment $r\in I^*$ is equivalent to the existence of a nonzero $c\in R$ such that $cr^q\in F^e(I)R$ for \emph{all} $q=p^e$ by \cite[Theorem 1.3(b)]{Hun96}.

A noetherian ring $R$ is called a \emph{splinter} if each module-finite map $R\to S$ inducing a surjection $\Spec S\twoheadrightarrow \Spec R$ admits an $R$-linear left inverse. Local splinters are always normal domains \cite[Lemma 2.1.1]{DT23}, and a noetherian domain $R$ is a splinter if and only if the map $R\to R^+$ is cyclically pure \cite[Lemma 2.3.1]{DT23}. Here, $R^+$ denotes the \emph{absolute integral closure} of $R$, i.e. the integral closure of $R$ in a choice of algebraic closure $\overline{K}$ of the fraction field $K=\Frac(R)$.

If $(R,\fm,k)$ is a local domain and $E:=E_R(k)$ is an injective hull of the residue field $k$ of $R$, then by \cite[Lemma 2.1(e)]{HH95} $R$ is a splinter if and only if $E\to E\otimes_R R^+$ is injective, provided that $R$ is approximately Gorenstein (e.g. if $R$ is Gorenstein or if $\widehat{R}$ is reduced). If $R$ has prime characteristic $p>0$, then similarly $R$ is $F$-pure if and only if $E\to E\otimes_R R^{1/p}$ is injective.

We now briefly recall the notion of a perfectoid ring following \cite{BIM19}, as well as that of perfectoid purity. Again let $p>0$ be a prime integer. The \emph{tilt} of a commutative ring $R$, denoted $R^\flat$, is the perfect ring of characteristic $p$ obtained via the inverse limit $\varprojlim R/pR$ taken over the maps $r\to r^p$. If $R$ is $p$-adically complete, then there is a unique lift $\theta$ of the natural projection $\przero:R^\flat\to R/pR$ yielding the diagram
\begin{equation*}
    \begin{tikzcd}
        W(R^\flat)\arrow[r,dotted,"\theta"]\arrow[d,twoheadrightarrow] & R\arrow[d,twoheadrightarrow]\\
        R^\flat\arrow[r,"\przero"]& R/pR
    \end{tikzcd}
\end{equation*}
where $W(-)$ denotes the ring of Witt vectors. Then $R$ is said to be \emph{perfectoid} if $R$ is $p$-adically complete, the Frobenius map $F:R/pR\to R/pR$ is surjective, $\ker(\theta)$ is a principal ideal of $W(R^\flat)$, and if there exists $\varpi\in R$ such that $\varpi^p = pu$ for some unit $u\in R$. The two primary examples of interest in this article come from: 
\begin{enumerate}[label=(\roman*)]
    \item \cite[Example 3.8(1)]{BIM19} A ring of prime characteristic $p>0$ is perfectoid if and only if it is perfect;\label{perfd-1}
    \item \cite[Example 3.8(2)]{BIM19} If $R$ is a local domain of mixed characteristic $(0,p)$, then the $p$-completion $(R^+)^{\land p}$ of the absolute integral closure is perfectoid.\label{perfd-2}
\end{enumerate}
Following \cite{BMPSTWW24} we then say that a local ring $(R,\fm)$ with $p\in \fm$ is \emph{perfectoid pure} if there exists a pure map $R\to B$ where $B$ is a perfectoid $R$-algebra. Note that in view of \ref{perfd-1}, a local ring $R$ of characteristic $p>0$ is perfectoid pure if and only if it is $F$-pure since the latter is equivalent to purity of $R\to \rperf$. The situation in the present article is somewhat simplified in that (almost) every perfectoid pure ring $R$ we consider arises either as one admitting a pure map $R\to (R^+)^{\land p}$ as in \ref{perfd-2} or as the $\fm$-adic completion of such a ring.

\section{Deformation of Gorenstein splinters, modulo \texorpdfstring{\cref{conj:CM}}{Conjecture 1.2}}
Given a complete local Gorenstein domain $(T,\fm)$ of mixed characteristic $(0,p)$ admitting an $F$-pure section, we can often find via the results of \cref{sec:Heitmann,sec:main} a local subring $A$ of $T$ whose completion is $T$ and which admits a weakly $F$-regular section.

\begin{lemma}\label{lem:splinter-injective-hull}
    Let $(R,\fm,k)$ be a Gorenstein local normal domain of dimension $d$. Then $R$ is a splinter if and only if the map $H^{d}_\fm(R)\to H^{d}_\fm(R^+)$ is injective.
\end{lemma}
\begin{proof}
    By \cite[Lemma 2.3.1]{DT23} $R$ is a splinter if and only if the map $R\to R^+$ is cyclically pure. Since $R$ is Gorenstein, this is equivalent to $R\to R^+$ being pure, e.g. by \cite[Theorem 2.6]{Hoc77}. Using that $R$ is Gorenstein once more, $H^d_\fm(R)$ serves as an injective hull of the residue field $k$. Let $\underline{x}=x_1,\dots, x_d$ be a system of parameters for $R$. Then $H^d_\fm(R^+)$ is a cokernel of the final map in the \u{C}ech complex on $\underline{x}$, so $H^d_\fm(R^+)\cong H^d_\fm(R)\otimes_R R^+$ since tensor products preserve cokernels. The conclusion then follows immediately from \cite[Lemma 2.1(e)]{HH95}.
\end{proof}

\begin{lemma}\label{lem:splinter lifts to splinter}
Let $(R,\fm,k)$ be a local Gorenstein domain of residual characteristic $\Char(k)=p$. If $H^{d-1}_\fm(R^+) = 0$ (for example, if \cref{conj:CM} holds) and if $x\in\fm$ is a nonzero divisor such that $R/xR$ is a splinter, then $R$ is a splinter.
\end{lemma}
\begin{proof}
    Note that $R/xR$ is a normal domain since it is a splinter \cite[Lemma 2.1.1]{DT23}, hence $R$ is a normal domain as well since normality deforms \cite[Proposition I.7.4]{Sey72}. The commutative diagram
\begin{equation}
\begin{tikzcd}
    0\arrow[r]&R\arrow[r,"\cdot x"]\arrow[d]&R\arrow[r]\arrow[d]& R/xR\arrow[r]\arrow[d]& 0\\
    0\arrow[r]&R^+\arrow[r,"\cdot x"]&R^+\arrow[r]& R^+/xR^+\arrow[r]& 0
\end{tikzcd}
\end{equation}
induces the diagram of local cohomology modules
\begin{equation}\label{eq:local-cohomology}
    \begin{tikzcd}
        0\arrow[r]&H^{d-1}_\fm(R/xR)\arrow[r]\arrow[d,"\xi"]& H^d_\fm(R)\arrow[r,"\cdot x"]\arrow[d,"\varphi"]&H^d_\fm(R)\arrow[d]\arrow[r]& 0\\
        0\arrow[r]&H^{d-1}_\fm(R^+/xR^+)\arrow[r]& H^d_\fm(R^+)\arrow[r,"\cdot x"]&H^d_\fm(R^+)\arrow[r]& 0
    \end{tikzcd}
\end{equation}
where the leftmost zeros are by assumption. Since $H^d_\fm(R)$ is artinian, we may check injectivity of the middle map in the diagram above at the socle. Let $\eta\in \ker\varphi\cap\Soc(H^d_\fm(R))$. Since $x\eta=0$, we may then view $\eta\in H^{d-1}_\fm(R/xR)$ by exactness of the top row.

By the integrality of the map $R\to R^+$, choose a prime $\fQ\in\Spec(R^+)$ lying over the prime $xR$. We then have the commutative diagram
\begin{equation}
    \begin{tikzcd}
        R\arrow[r]\arrow[d]& R/xR\arrow[r,"\subseteq"]\arrow[d,dotted]& R^+/\fQ \arrow[r,"\cong"]&(R/xR)^+\\
        R^+\arrow[r]& R^+/xR^+ \arrow[urr,twoheadrightarrow]
    \end{tikzcd}
\end{equation}
where the isomorphism follows from the proof of \cite[Proposition 1.2]{HH95}. We then obtain the diagram
\begin{equation}
    \begin{tikzcd}
        H^{d-1}_\fm(R/xR)\arrow[rr,hookrightarrow]\arrow[dr,"\xi"']&&H^{d-1}_\fm((R/xR)^+)\\
        &H^{d-1}_\fm(R^+/xR^+)\arrow[ur]
    \end{tikzcd}
\end{equation}
where the top map is injective by \cref{lem:splinter-injective-hull} since $R/xR$ is a splinter. It follows that $\xi$ is injective, and commutativity of the diagram \eqref{eq:local-cohomology} tells us that $\eta=0$ so that $R$ is a splinter by \cref{lem:splinter-injective-hull} once more.
\end{proof}

\section{A Heitmann-type completion lemma}\label{sec:Heitmann}
This section is the technical heart of the paper. Starting with a complete local domain $T$ of mixed characteristic $(0,p)$ in which $pT$ is a prime ideal, we design a local subring $A$ whose completion is $T$ with some specific properties. When combined with certain data of a Gorenstein $F$-pure section of $T$, \cref{conj:CM} will imply that this subring $A$ can be manufactured to be a splinter. The main theorem of this section is:

\begin{theorem}(\emph{cf.} \cite[Theorem 22]{LR01}\label{thm:precompletion})
    Let $(T,\fm)$ be a complete local domain of mixed characteristic $(0,p)$ such that $(p)\in\Spec T$, and let $\fp\in\Spec T\setminus\{\fm\}$ be a non-maximal prime ideal such that $\fp\cap \Z = (p)$. Let $\delta\in \fm\setminus\fp$. Then there exists a local domain $(A,\fm\cap A)$ such that $\widehat{A}\cong T$, $A\cap \fp = pA$, and $\delta t\in A$ for some unit $t\in T$.
\end{theorem}

To prove this theorem we first adapt a version of a definition from \cite{DJLPR07}.
\begin{definition}[cf. {\cite[Definition 3.1]{DJLPR07}}]\label{def:px-subring}
    Let $(T,\fm,k)$ be a complete local domain and let $x\in T$ be a nonzero regular element of $T$ such that $xT\in\Spec T$. Let $\fp\in \Spec^\circ T$ be a nonmaximal prime ideal, and suppose that $(R,\fm\cap R)$ is a quasi-local subring of $T$ satisfying the following two properties:
    \begin{enumerate}[label=(\alph*)]
        \item $|R|<|T|$; \label{def:subring-1}
        \item If $\fq\in\Spec T$ where $\fq\subseteq \fp$ and $x\not\in\fq$, then $R\cap \fq = (0)$. \label{def:subring-2}
    \end{enumerate}
    Then we call $R$ a \emph{$[\fp,x]$-subring of $T$.}
\end{definition}

We include a list of facts about $[\fp,x]$-subrings taken from \cite{DJLPR07}.
\begin{theorem}\label{thm-list} Let $T$, $\fp\in\Spec^\circ(T)$, and $x\in \fp$ be as in the setting of \cref{def:px-subring}.
    \begin{enumerate}[label=(\roman*)]
        \item \cite[Lemma 3.5]{DJLPR07} For any $[\fp,x]$-subring $R\subseteq T$, there exists an intermediate $[\fp,x]$-subring $R\subseteq S\subseteq T$ such that $xT \cap S = xS$ and $|S|\leq \sup(\aleph_0,|R|)$;\label{thm-list-1}
        \item \cite[Corollary 3.6]{DJLPR07} If $R\subseteq T$ is a $[\fp,x]$-subring such that $xT\cap R = xR$, then $\fp\cap R = xR$;\label{thm-list-2}
        \item \cite[Lemma 3.11]{DJLPR07} For every $[\fp,x]$-subring $R\subseteq T$ such that $xT\cap R = xR$ and every $u\in T$, there exists a $[\fp,x]$-subring $S\subseteq T$ such that $R\subseteq S\subseteq T$, $u+\fm^2$ is in the image of $S\to T/\fm^2$, and $IT\cap S = I$ for every finitely generated ideal $I\subseteq S$. Moreover, $|S| = \sup(\aleph_0,|R|)$.\label{thm-list-5}
    \end{enumerate}
\end{theorem}

\begin{lemma}\label{lem:px-subring-tower}
    Let $T$, $\fp\in\Spec^\circ(T)$, and $x\in \fp$ be as in the setting of \cref{def:px-subring}. Let $y\in\fm\setminus \fp$, and let $A$ be a $[\fp,x]$-subring of $T$. Then there exists an infinite $[\fp,x]$-subring $S$ of $T$ such that $A\subseteq S\subseteq T$, $y t\in S$ for some unit $t\in T$, and $xT\cap S = xS$. Moreover, $|S|\leq \sup(\aleph_0,|R|)$.
\end{lemma}
\begin{proof}
By \cref{thm-list}\ref{thm-list-1} we can find a $[\fp,x]$-subring $R\subseteq T$ such that $A\subseteq R\subseteq T$ and such that $xT\cap R = xR$. Note that $\fp\cap R = xR$ as well by \cref{thm-list}\ref{thm-list-2}. Let $D\subseteq T/\fp$ denote a full set of coset representatives of the elements $u + \fp$ such that $u$ is algebraic over $R/xR$. By \cite[Lemma 16]{LR01} if $R$ is countable and \cite[Lemma 4]{Loe97} if $R$ is uncountable, there exists a unit $t\in T$ such that $$yt\notin\bigcup\{\fp+r\mid r\in D\}.$$

We claim that $S:= R[yt]_{\fm\cap R[yt]}$ is a $[\fp,x]$-subring of $T$. The first step is to show that, for all $\fq\subseteq \fp$ with $x\notin \fq$, the element $yt + \fq$ is transcendental over $R/(\fq\cap R) \cong R$. To see this, suppose $yt + \fq$ satisfies the polynomial equation 
\[
c_n(yt + \fq)^n + \dots  + c_0 = 0,
\]
where $c_i\in R$ for $0\leq i\leq n$. Then the same equation holds for $yt + \fp$ over $R/(\fp\cap R)$. Since $\fp\cap R = xR$ by \Cref{thm-list}\ref{thm-list-2}, so $R/(\fp\cap R)\cong R/xR$ and we have the following relation in $R/xR$.
\[
\bar{c_n}(yt + \fp)^n + \dots + \bar{c_0} = 0
\]
Since $yt + \fp$ is transcendental over $R/xR$, we have that $\bar{c_i} = 0$ for all $0\leq i\leq n$. We may therefore write $c_i = x\frac{c_i}{x}$, and $yt + \fq$ satisfies the following relation over $R$:
\[
\frac{c_n}{x}(yt + \fq)^n + \dots + \frac{c_0}{x} = 0.
\]
As $\bigcap_{m>0}x^mT = (0)$, we conclude that $c_i = 0$ for all $0\leq i\leq n$, hence $yt + \fq$ is transcendental over $R$. 

Suppose now that $f\in R[yt]\cap \fq$, and write $f = c_n(yt)^n +\dots + c_0$ where $c_i\in R$. Since $yt + \fq$ is transcendental over $R$, we have $c_i\in R\cap \fq = (0)$ for all $1\leq i\leq n$, so $f = c_0\in R\cap \fq = (0)$. It follows that $S:= R[yt]_{\fm\cap R[yt]}$ is a $[\fp, x]$-subring of $T$ containing $yt$. We may then replace $S$ by the one given by \cref{thm-list}\ref{thm-list-1} to conclude that $xT\cap S = xS$.
\end{proof}

The subring $A$ of $T$ which we construct in the proof of \cref{thm:completion} will be a union of $[\fp,p]$-subrings of $T$. To show that $A$ is noetherian, we employ the following observation of Heitmann.

\begin{proposition}[{\cite[Proposition 1]{Hei93}}]\label{thm:Heitmann}
    Let $(R,\fm\cap R)$ be a quasi-local subring of a complete local ring $(T,\fm)$. If the map $R\to T/\fm^2$ is surjective and $IT\cap R = I$ for every finitely generated ideal $I\subseteq R$, then $R$ is noetherian and the natural homomorphism $\widehat{R}\to T$ is an isomorphism.
\end{proposition}

Given a well-ordered set $\Omega$ and an element $\alpha\in\Omega$, define $\gamma(\alpha):=\sup\{\beta\in\Omega\mid \beta<\alpha\}$.

\begin{proof}[Proof of \cref{thm:precompletion}]
    Let $\Omega = T/\fm^2$ and well-order $\Omega$ so that $0$ is its first element and so that each element of $\Omega$ has fewer than $|\Omega|$ predecessors. Note that $\Z_{(p)}$ is a $[\fp,p]$-subring of $T$. Let $R_0$ be the $[\fp,p]$-subring granted by \cref{lem:px-subring-tower} for $x=p$ such that $\Z_{(p)}\subseteq R_0\subseteq T$, $\delta t\in R_0$ for some unit $t\in T$, and such that $pT\cap R_0 = pR_0$. Note that $|R_0| = \aleph_0$. Now for each $\alpha\in\Omega$, define $R_\alpha$ recursively as follows. If $\gamma(\alpha)<\alpha$, then define $R_\alpha$ to be the $[\fp,p]$-subring granted by applying \cref{thm-list}\ref{thm-list-5} to $x=p$; that is, $R_{\gamma(\alpha)}\subseteq R_\alpha\subseteq T$, $\gamma(\alpha)$ is in the image of the map $R_\alpha\to T/\fm^2$, and $IT\cap R_\alpha = I$ for all finitely generated ideals $I\subseteq R_\alpha$. Note that in particular $pT\cap R_\alpha = pR_\alpha$.
    
    If $\gamma(\alpha) = \alpha$, then instead let $R_\alpha:=\cup_{\beta<\alpha} R_\beta$. By the above discussion, $pT\cap R_\alpha = pR_\alpha$ in this case as well. We show that $R_\alpha$ is a $[\fp,x]$-subring; condition \ref{def:subring-2} is clear, so we show that $|R_\alpha|<|T|$. We first show
\begin{align}
    |R_\alpha|\leq \sup(\aleph_0,|\{\beta\in\Omega\mid \beta<\alpha\}|)\label{thm:precompletion-1}
\end{align}
and we accomplish this via transfinite induction, reproducing the details of \cite[Lemma 6]{Hei93} or \cite[Lemma 3.12]{DJLPR07} for the reader's convenience. The case of $\alpha=0$ is clear, so suppose that $\alpha\in \Omega$ and that
\begin{align*}
    |R_\beta|\leq\sup(\aleph_0,|\{\lambda\in\Omega\mid \lambda<\beta\}|)
\end{align*}
for every $\beta<\alpha$. If $\gamma(\alpha)<\alpha$ then
\begin{align*}
    |R_\alpha| & = \sup(\aleph_0,|R_{\gamma(\alpha)}|)\\
     & \leq \sup(\aleph_0,\sup(\aleph_0,|\{\lambda\in\Omega\mid \lambda<\gamma(\alpha)\}|))\\
     & \leq \sup(\aleph_0,|\{\beta\in\Omega\mid \beta<\alpha\}|)
\end{align*}
while if $\gamma(\alpha) = \alpha$ then
\begin{align*}
    |R_\alpha| &\leq\sum\limits_{\beta<\alpha} |R_\beta|\\
    & \leq |\{\beta\in\Omega\mid \beta<\alpha\}|\cdot \sup_{\beta<\alpha}(|R_\beta|)\\
    & \leq |\{\beta\in\Omega\mid \beta<\alpha\}|\cdot \sup_{\beta<\alpha}(\sup(\aleph_0,|\{\lambda\in\Omega|\lambda<\beta\}|))\\
    & \leq |\{\beta\in\Omega\mid \beta<\alpha\}|\cdot \sup(\aleph_0,|\{\beta\in\Omega\mid \beta<\alpha\}|)\\
    & = \sup(\aleph_0,|\{\beta\in\Omega\mid \beta<\alpha\}|)
\end{align*}
and \eqref{thm:precompletion-1} follows. The ordering of $\Omega$ that we chose at the outset shows that $|R_\alpha|<|T|$ for every $\alpha\in \Omega$, so each $R_\alpha$ is a $[\fp,p]$-subring as claimed.

Finally let $A:=\cup_{\alpha\in\Omega} R_\alpha$, which we claim is the desired ring. The map $A\to T/\fm^2$ is surjective by construction. Since $pT\cap R_\alpha = pR_\alpha$ for each $\alpha\in\Omega$ we have $\fp\cap R_\alpha = pR_\alpha$ by \cref{thm-list}\ref{thm-list-2}, hence $\fp\cap A = pA$. Since $IT\cap R_\alpha = I$ for each $\alpha$ and for each finitely generated ideal $I\subseteq R_\alpha$, the same is true for finitely generated ideals of $A$. It follows from \cref{thm:Heitmann} that $A$ is noetherian and $\widehat{A}\cong T$.
\end{proof}

\section{Proof of \texorpdfstring{\cref{thm:main-theorem-1,thm:main-theorem-2}}{Theorems A and B}}\label{sec:main}

In this section, if $R$ has mixed characteristic $(0,p)$ and $r\in R$ then $\overline{r}$ denotes its image in $R/pR$.
\begin{definition}
    Let $A$ be a ring of characteristic $p>0$. Let $I\subseteq A$ be an ideal and $x\in A$ an element. For each $e>0$ define $I^{[p^e]} := F^e(I)A$ and 
    \[
    \mathcal{M}_{I,x}:= \bigcup_{e\geq 0}\bigcap_{f\geq e} (I^{[p^f]}: x^{p^f}).
    \]
\end{definition}
If $(A,\fm)$ is local and Gorenstein, it is straightforward to check that $A$ is $F$-pure if and only if $\mathcal{M}_{I,\delta}\subseteq\fm$ for every (equivalently, for some) parameter ideal $I\subseteq A$ and socle generator $\delta$ of $A/I$.

\begin{theorem}\label{thm:completion}
    Let $(T,\fm)$ be a complete local ring with $d=\dim(T)\geq 2$ and of mixed characteristic $(0, p)$ such that $T/pT$ is an $F$-pure local Gorenstein domain. Let $\fa = (p,x_2,\dots, x_d)\subseteq T$ be a parameter ideal and denote $\fb = \fa/(p)$, a parameter ideal in $T/pT$. Let $\delta\in T$ be a generator for the socle of $T/\fa$. Further assume that there is a nonmaximal prime ideal $\fp\in\Spec T$ with the following properties:
    \begin{enumerate}[label=(\alph*)]
        \item $\fp\cap \Z=p\Z$,\label{thm:completion-1}
        \item $\delta \notin \fp$,\label{thm:completion-2}
        \item $\mathcal{M}_{\fb, \overline{\delta}} \subseteq \frac{\fp}{pT}$.\label{thm:completion-3}
    \end{enumerate}
    Then there exists a local domain $(A,\fm\cap A)$ of mixed characteristic $(0,p)$ such that $\widehat{A}\cong T$, $pT\cap A =pA$, and such that $A/pA$ is weakly $F$-regular.
\end{theorem}
\begin{proof}
   Since $\delta$ generates the socle of $T/\fa$ and $T/pT$ is $F$-pure, it follows that $\delta\notin pT$ and that $\overline{\delta}^q\notin\fb^{[q]}$ for all $q=p^e$. Use \cref{thm:precompletion} to construct a local domain $(A,\fm\cap A)$ such that $\widehat{A}\cong T$, $\delta t \in A$ for some unit $t\in T$, and $ \fp\cap A = pA$. Let $I\subseteq A$ be a parameter ideal such that $IT = \fa$. Note that $p\in I$. We will show first that $B:=A/pA$ is weakly $F$-regular and it suffices by \cite[Theorem 1.5]{Hun96} to show that $J:=I/pA$ is tightly closed since $B$ is Gorenstein.

    Note that $B$ is a domain and $\overline{\delta t}$ generates the socle of $B/J$. Since $B/J$ is artinian Gorenstein, its socle is contained in every nonzero ideal. To show $J=J^*$, it therefore suffices to show $\overline{\delta t}\notin J^*$. Suppose instead that $\overline{\delta t}\in J^*$. Then there is some nonzero $\overline{c}\in B$ such that $$\overline{c}(\overline{\delta t})^q \in J^{[q]}\subseteq \fb^{[q]}\subseteq T/pT$$ for all $q=p^e$. Since $\mathcal{M}_{\fb, \overline{\delta t}}=\mathcal{M}_{\fb,\overline{\delta}}$, we observe by the above and condition \ref{thm:completion-3} that
    \[
    \overline{c}\in \mathcal{M}_{\fb, \overline{\delta}}\cap B \subseteq \frac{\fp}{pT}\cap B=(0),
    \]
a contradiction. Hence $\overline{\delta t}\notin J^*$, so $B$ is weakly $F$-regular as desired.
\end{proof}
In order to prove \cref{thm:main-theorem-1}, we apply \cref{thm:completion} to a class of rings which need not satisfy assumption \ref{thm:completion-3}. If $T$ satisfies conditions \ref{thm:completion-1} and \ref{thm:completion-2} but not necessarily \ref{thm:completion-3}, we show that indeed the power series ring $T\llbracket w\rrbracket$ with the nonmaximal ideal $\fp=\fm T\llbracket w\rrbracket$ satisfies all three. We will need the following limitation imposed on the ideals $\mathcal{M}_{I,g}$ in a power series ring.

\begin{lemma}\cite[Proposition 3]{LR01}\label{lem:colon-power-series}
    Let $(R,\fm)$ be a reduced local ring of prime characteristic $p>0$, let $I\subseteq R\llbracket w\rrbracket$ be an ideal, and let $g\in R\llbracket w\rrbracket$. Then either $\mathcal{M}_{I,g} = R\llbracket w\rrbracket$ or $\mathcal{M}_{I,g}\subseteq \fm R\llbracket w\rrbracket$.
\end{lemma}

\begin{corollary}\label{cor:splinter-completion-power-series}
    Let $(T,\fm)$ be a complete local ring with $d = \dim(T) \geq 2$ and of mixed characteristic $(0,p)$ such that $T/pT$ is an $F$-pure Gorenstein domain. Then there exists a local subring $(A,\fm\cap A)$ of $T$ such that $\widehat{A}\cong T\llbracket w\rrbracket$, $pT\llbracket w\rrbracket\cap A = pA$, and such that $A/pA$ is weakly $F$-regular.
\end{corollary}
\begin{proof}
    Let $\fa = (p,x_2,\dots, x_d)\subseteq T$ be a parameter ideal and note that $(\fa, w)$ is a parameter ideal of $T\llbracket w\rrbracket$. Denote $\fb:=\fa/pT$ and $\fc:=(\fa,w)/(p)$. If $\delta\in T$ generates the socle of $T/\fa$, then $\delta\notin pT$ and $\delta+w$ generates the socle of $T\llbracket w\rrbracket/(\fa,w)$. Consider the prime ideal $\fp := \fm T\llbracket w\rrbracket$. Note that $\delta+w\notin\fp$, and $\overline{\delta+w}^{p^e}\notin\fc^{[p^e]}$ for all $e\geq 1$ since we may ascend $F$-purity along $T/pT\to T/pT\llbracket w\rrbracket \cong T\llbracket w\rrbracket/pT\llbracket w\rrbracket$ by \cite[Theorem 7.4]{MP25}. Consequently, for all $e > 0$ we have $(\fc^{[p^e]}:(\overline{\delta+w})^{p^e})\subseteq \frac{(\fm, w)}{pT}$, hence $\mathcal{M}_{\fc,\overline{\delta+w}}\subseteq \frac{(\fm, w)}{pT}$. By \cref{lem:colon-power-series}, we in fact have $\mathcal{M}_{\fc,\overline{\delta+w}}\subseteq \frac{\fp}{pT}$. We may now apply \cref{thm:completion} to $T\llbracket w\rrbracket$ together with the parameter ideal $(\fa, w)$ and the non-maximal prime ideal $\fp$ to obtain the result.
\end{proof}
\begin{remark}
    The proof of \cref{cor:splinter-completion-power-series} does not adapt in a straightforward way to the case where $\pi\in T$ is an arbitrary nonzero divisor such that $T/\pi T$ is a Gorenstein $F$-pure domain. The reason is that if $T/\pi T$ is a domain of  characteristic $p > 0$, then $T\llbracket w\rrbracket$ admits a $[\fm T\llbracket w \rrbracket, \pi]$-subring if and only if  $\pi\in \sqrt{(p)}$.  
\end{remark}

In the proof of the next theorem alone, we use $(-)^{\land \fm}$ for $\fm$-adic completion and $(-)^{\land p}$ for $p$-completion.
\begin{theorem}\label{thm:splinter-implies-perfd-pure}
    If $(R,\fm,k)$ is a local splinter of mixed characteristic $(0,p)$, then $R$ is perfectoid pure.
\end{theorem}
\begin{proof}
    Let $E=E_R(k)$ denote an injective hull of the residue field, and let $(-)^\vee:=\Hom_R(-,E)$ denote the Matlis dual. $R\to R^+$ is pure by assumption, so $E\to E\otimes_R R^+$ is injective by \cite[Lemma 2.1(e)]{HH95}. Applying $(-)^\vee$ and using $\Hom$-$\otimes$ adjunction, we have an $R$-module surjection
\begin{align}
    \Hom_R(R^+,R^{\land \fm})&\cong \Hom_R(R^+,\Hom_R(E,E))\\
     & \cong \Hom_R(E\otimes_R R^+,E)\twoheadrightarrow\Hom_R(E,E)\cong R^{\land \fm}\cong \Hom_R(R,R^{\land \fm}).
\end{align}
It follows that the $\fm$-adic completion map $R\to R^{\land \fm}$ extends to a map $\psi: R^+\to R^{\land \fm}$, hence the $\fm$-adic completion map factors as follows.
\begin{equation*}
    \begin{tikzcd}
        R\arrow[r,"\varphi"]\arrow[rr,bend right] & R^+\arrow[r, "\psi"] & R^{\land \fm}
    \end{tikzcd}
\end{equation*}
Completing again, the identity map on $R^{\land \fm}$ factors as $R^{\land \fm}\xrightarrow{\varphi^{\land \fm}} (R^+)^{\land \fm}\to R^{\land \fm}$, so $\varphi^{\land\fm}$ is pure. Consider the commutative diagram
\begin{equation*}
    \begin{tikzcd}
        R\arrow[r]\arrow[dr]& R^{\land \fm}\arrow[r,"\varphi^{\land \fm}"]& (R^+)^{\land\fm}\\
        & (R^+)^{\land p}\arrow[ur]&
    \end{tikzcd}
\end{equation*}
whose top row is pure. It follows that $R\to (R^+)^{\land p}$ is pure, the target of which is a perfectoid ring by \cite[Example 3.8(2)]{BIM19}. Then $R$ is perfectoid pure, as desired.
\end{proof}

\begin{theorem}\label{thm:perfd-pure-deformation}
    Let $(R,\fm,k)$ be a $d$-dimensional complete local ring of mixed characteristic $(0,p)$, and suppose that $R/pR$ is an $F$-pure Gorenstein domain. If $H^d_\fn(A^+)=0$ for every local subring $(A,\fn)\subseteq R\llbracket x\rrbracket$ such that $\widehat{A}\cong R\llbracket x\rrbracket$ (for example, if \cref{conj:CM} holds), then $R$ is perfectoid pure.
\end{theorem}
\begin{proof}
   If $\dim R<2$ then $R$ is regular. Hence we may assume that we are in the setting of \cref{cor:splinter-completion-power-series}, so there exists a local subring $A\subseteq T\llbracket w\rrbracket$ such that $\widehat{A}\cong R\llbracket w\rrbracket$ and such that $A/pA$ is weakly $F$-regular (hence a splinter by \cite[Remark 2.4.1(2)]{DT23}). By assumption, we may apply \cref{lem:splinter lifts to splinter} to obtain that $A$ is a splinter and hence perfectoid pure by \cref{thm:splinter-implies-perfd-pure}. It follows that $R\llbracket w\rrbracket$ is perfectoid pure by \cite[Lemma 4.8]{BMPSTWW24}, so we may conclude that $R$ is perfectoid pure by \cite[Lemma 4.6]{BMPSTWW24}.
\end{proof}

Note that the converse of \cref{thm:perfd-pure-deformation} fails, even in the complete intersection case. Indeed, if $p\equiv 2\mod 3$, then the ring $$T = \frac{\Z_p\llbracket x,y,z\rrbracket}{(x^3+y^3+z^3)}$$ is perfectoid pure \cite[Proposition 6.10]{Yos25} whereas $T/pT$ is not $F$-pure.

We now conclude with a proof of \cref{thm:main-theorem-2}.
\begin{theorem}\label{thm:conj1 iff conj2}
    The following two assertions are equivalent.
\begin{enumerate}[label=(\Roman*)]
    \item Let $(R,\fm)$ be a complete local Gorenstein ring of mixed characteristic $(0,p)$. If $R/pR$ is an $F$-pure domain, then $R$ is perfectoid pure.\label{thm:B-2-1}
    \item Let $(R,\fm)$ be a local Gorenstein ring of mixed characteristic $(0,p)$. If $R/pR$ is weakly $F$-regular and analytically irreducible, then $R$ is perfectoid pure.\label{thm:B-2-2}
\end{enumerate}
\end{theorem}
\begin{proof}
    The backward direction is identical to the proof of \cref{thm:perfd-pure-deformation} with \ref{thm:B-2-2} in place of \cref{lem:splinter lifts to splinter} and \cref{thm:splinter-implies-perfd-pure}. For the forward direction, suppose that $R$ is as in \ref{thm:B-2-2} and consider the following diagram.
\begin{equation*}
    \begin{tikzcd}
        R\arrow[r]\arrow[d]& \widehat{R}\arrow[d]\\
        R/pR\arrow[r] & \widehat{R}/p \widehat{R}
    \end{tikzcd}
\end{equation*}
$R/pR$ is $F$-pure by \cite[Remark 1.6]{FW89} so $\widehat{R}/p\widehat{R}$ is an $F$-pure domain by assumption and \cite[Corollary 2.3]{MP25}. By \ref{thm:B-2-1} we have that $\widehat{R}$ is perfectoid pure, hence so is $R$ by \cite[Lemma 4.8]{BMPSTWW24} as desired.
\end{proof}

\printbibliography

@article {LR01,
    AUTHOR = {Loepp, S. and Rotthaus, C.},
     TITLE = {Some results on tight closure and completion},
   JOURNAL = {J. Algebra},
  FJOURNAL = {Journal of Algebra},
    VOLUME = {246},
      YEAR = {2001},
    NUMBER = {2},
     PAGES = {859--880},
      ISSN = {0021-8693,1090-266X},
   MRCLASS = {13A35},
  MRNUMBER = {1872128},
MRREVIEWER = {Florian\ Enescu},
       DOI = {10.1006/jabr.2001.9006},
}

@article {Loe97,
    AUTHOR = {Loepp, S.},
     TITLE = {Constructing local generic formal fibers},
   JOURNAL = {J. Algebra},
  FJOURNAL = {Journal of Algebra},
    VOLUME = {187},
      YEAR = {1997},
    NUMBER = {1},
     PAGES = {16--38},
      ISSN = {0021-8693,1090-266X},
   MRCLASS = {13B35 (13B24 13C15 13J10)},
  MRNUMBER = {1425557},
MRREVIEWER = {Sylvia\ Wiegand},
       DOI = {10.1006/jabr.1997.6768},
}

@misc{Bha21,
      title={Cohen-Macaulayness of absolute integral closures}, 
      author={Bhargav Bhatt},
      year={2021},
      eprint={2008.08070},
eprinttype = {arxiv},
}

@article{Quy16,
    AUTHOR = {Quy, Pham Hung},
     TITLE = {On the vanishing of local cohomology of the absolute integral
              closure in positive characteristic},
   JOURNAL = {J. Algebra},
  FJOURNAL = {Journal of Algebra},
    VOLUME = {456},
      YEAR = {2016},
     PAGES = {182--189},
      ISSN = {0021-8693,1090-266X},
   MRCLASS = {13D45 (13A35 13B40 13D22 13H10)},
  MRNUMBER = {3484140},
MRREVIEWER = {Linquan\ Ma},
       DOI = {10.1016/j.jalgebra.2016.02.017},
}

@misc{MP25,
Author = {Linquan Ma and Thomas Polstra},
Title = {{$F$}-singularities: a commutative algebra approach},
Year = {2025},
journal = {preprint},
url = {https://www.math.purdue.edu/~ma326/F-singularitiesBook.pdf}
}

@article {Hoc77,
    AUTHOR = {Hochster, Melvin},
     TITLE = {Cyclic purity versus purity in excellent {N}oetherian rings},
   JOURNAL = {Trans. Amer. Math. Soc.},
  FJOURNAL = {Transactions of the American Mathematical Society},
    VOLUME = {231},
      YEAR = {1977},
    NUMBER = {2},
     PAGES = {463--488},
      ISSN = {0002-9947,1088-6850},
   MRCLASS = {13D99},
  MRNUMBER = {463152},
MRREVIEWER = {Toma\ Albu},
       DOI = {10.2307/1997914},
}

@article {HH92,
    AUTHOR = {Hochster, Melvin and Huneke, Craig},
     TITLE = {Infinite integral extensions and big {C}ohen-{M}acaulay
              algebras},
   JOURNAL = {Ann. of Math. (2)},
  FJOURNAL = {Annals of Mathematics. Second Series},
    VOLUME = {135},
      YEAR = {1992},
    NUMBER = {1},
     PAGES = {53--89},
      ISSN = {0003-486X,1939-8980},
   MRCLASS = {13H10 (13A02 13A35 13B22)},
  MRNUMBER = {1147957},
MRREVIEWER = {Liam\ O'Carroll},
       DOI = {10.2307/2946563},
}

@article{BMPSTWW24,
Author = {Bhargav Bhatt and Linquan Ma and Zsolt Patakfalvi and Karl Schwede and Kevin Tucker and Joe Waldron and Jakub Witaszek},
Title = {Perfectoid pure singularities},
Year = {2024},
Eprint = {2409.17965},
eprinttype = {arxiv},
}

@book {Hun96,
    AUTHOR = {Huneke, Craig},
     TITLE = {Tight closure and its applications},
    SERIES = {CBMS Regional Conference Series in Mathematics},
    VOLUME = {88},
      NOTE = {With an appendix by Melvin Hochster},
 PUBLISHER = {Conference Board of the Mathematical Sciences, Washington, DC;
              by the American Mathematical Society, Providence, RI},
      YEAR = {1996},
     PAGES = {x+137},
      ISBN = {0-8218-0412-X},
   MRCLASS = {13-02 (13A35 13C14 13H10)},
  MRNUMBER = {1377268},
MRREVIEWER = {Ian\ M.\ Aberbach},
       DOI = {10.1090/cbms/088},
}

@incollection {Sey72,
    AUTHOR = {Seydi, Hamet},
     TITLE = {La th\'eorie des anneaux japonais},
 BOOKTITLE = {Colloque d'{A}lg\`ebre {C}ommutative ({R}ennes, 1972)},
     PAGES = {Exp. No. 12, 82},
 PUBLISHER = {Univ. Rennes, Rennes},
      YEAR = {1972},
   MRCLASS = {13B20},
  MRNUMBER = {366896},
MRREVIEWER = {M.\ Nagata},
}

@article {DJLPR07,
    AUTHOR = {Dundon, A. and Jensen, D. and Loepp, S. and Provine, J. and
              Rodu, J.},
     TITLE = {Controlling formal fibers of principal prime ideals},
   JOURNAL = {Rocky Mountain J. Math.},
  FJOURNAL = {The Rocky Mountain Journal of Mathematics},
    VOLUME = {37},
      YEAR = {2007},
    NUMBER = {6},
     PAGES = {1871--1891},
      ISSN = {0035-7596,1945-3795},
   MRCLASS = {13B35 (13J10)},
  MRNUMBER = {2382631},
MRREVIEWER = {Roger\ A.\ Wiegand},
       DOI = {10.1216/rmjm/1199649827},
}

@article {HL07,
    AUTHOR = {Huneke, Craig and Lyubeznik, Gennady},
     TITLE = {Absolute integral closure in positive characteristic},
   JOURNAL = {Adv. Math.},
  FJOURNAL = {Advances in Mathematics},
    VOLUME = {210},
      YEAR = {2007},
    NUMBER = {2},
     PAGES = {498--504},
      ISSN = {0001-8708,1090-2082},
   MRCLASS = {13A35 (13B22 13D45 13H10)},
  MRNUMBER = {2303230},
MRREVIEWER = {Mark\ R.\ Johnson},
       DOI = {10.1016/j.aim.2006.07.001},
}

@article{BIM19,
    AUTHOR = {Bhatt, Bhargav and Iyengar, Srikanth B. and Ma, Linquan},
     TITLE = {Regular rings and perfect(oid) algebras},
   JOURNAL = {Comm. Algebra},
  FJOURNAL = {Communications in Algebra},
    VOLUME = {47},
      YEAR = {2019},
    NUMBER = {6},
     PAGES = {2367--2383},
      ISSN = {0092-7872,1532-4125},
   MRCLASS = {13A35 (13D05 13D22 13H05 14G45)},
  MRNUMBER = {3957103},
MRREVIEWER = {Adela\ N.\ Vraciu},
       DOI = {10.1080/00927872.2018.1524009},
}

@article{LS25,
    AUTHOR = {Loepp, S. and Simpson, Austyn},
     TITLE = {Noncatenary splinters in prime characteristic},
   JOURNAL = {J. Algebra},
  FJOURNAL = {Journal of Algebra},
    VOLUME = {677},
      YEAR = {2025},
     PAGES = {61--87},
      ISSN = {0021-8693,1090-266X},
   MRCLASS = {99-06},
  MRNUMBER = {4889183},
       DOI = {10.1016/j.jalgebra.2025.03.029},
}

@article {PS23,
    AUTHOR = {Polstra, Thomas and Simpson, Austyn},
     TITLE = {{$F$}-purity deforms in {$\mathbb{Q}$}-{G}orenstein rings},
   JOURNAL = {Int. Math. Res. Not. IMRN},
  FJOURNAL = {International Mathematics Research Notices. IMRN},
      YEAR = {2023},
    NUMBER = {24},
     PAGES = {20725--20747},
      ISSN = {1073-7928,1687-0247},
   MRCLASS = {13A35},
  MRNUMBER = {4681270},
MRREVIEWER = {William\ D.\ Taylor},
       DOI = {10.1093/imrn/rnac254},
}

@article{Yos25,
Author = {Shou Yoshikawa},
Title = {Computation method for perfectoid purity and perfectoid BCM-regularity},
Year = {2025},
eprinttype = {arxiv},
Eprint = {2502.06108},
}

@incollection {Lec86,
    AUTHOR = {Lech, Christer},
     TITLE = {A method for constructing bad {N}oetherian local rings},
 BOOKTITLE = {Algebra, algebraic topology and their interactions
              ({S}tockholm, 1983)},
    SERIES = {Lecture Notes in Math.},
    VOLUME = {1183},
     PAGES = {241--247},
 PUBLISHER = {Springer, Berlin},
      YEAR = {1986},
      ISBN = {3-540-16453-7},
   MRCLASS = {13B35 (13C15 13H99)},
  MRNUMBER = {846452},
MRREVIEWER = {Luchezar\ L.\ Avramov},
       DOI = {10.1007/BFb0075463},
}

@incollection {PST25,
    AUTHOR = {Polstra, Thomas and Simpson, Austyn and Tucker, Kevin},
     TITLE = {On {$F$}-pure inversion of adjunction},
 BOOKTITLE = {Higher dimensional algebraic geometry---a volume in honor of
              {V}. {V}. {S}hokurov},
    SERIES = {London Math. Soc. Lecture Note Ser.},
    VOLUME = {489},
     PAGES = {319--344},
 PUBLISHER = {Cambridge Univ. Press, Cambridge},
      YEAR = {2025},
      ISBN = {978-1-009-39624-0},
   MRCLASS = {13A35},
  MRNUMBER = {4844638},
DOI = {10.1017/9781009396233.019}
}

@article {Sin99a,
    AUTHOR = {Singh, Anurag K.},
     TITLE = {Deformation of {$F$}-purity and {$F$}-regularity},
   JOURNAL = {J. Pure Appl. Algebra},
  FJOURNAL = {Journal of Pure and Applied Algebra},
    VOLUME = {140},
      YEAR = {1999},
    NUMBER = {2},
     PAGES = {137--148},
      ISSN = {0022-4049,1873-1376},
   MRCLASS = {13A35 (13H10)},
  MRNUMBER = {1693967},
MRREVIEWER = {Ian\ M.\ Aberbach},
       DOI = {10.1016/S0022-4049(98)00014-0},
}

@article{Sin99b,
    AUTHOR = {Singh, Anurag K.},
     TITLE = {{$F$}-regularity does not deform},
   JOURNAL = {Amer. J. Math.},
  FJOURNAL = {American Journal of Mathematics},
    VOLUME = {121},
      YEAR = {1999},
    NUMBER = {4},
     PAGES = {919--929},
      ISSN = {0002-9327,1080-6377},
   MRCLASS = {13A35 (13C40 13H10)},
  MRNUMBER = {1704481},
MRREVIEWER = {Ian\ M.\ Aberbach},
DOI = {10.1353/ajm.1999.0029},
}

@article {HH95,
    AUTHOR = {Hochster, Melvin and Huneke, Craig},
     TITLE = {Applications of the existence of big {C}ohen-{M}acaulay
              algebras},
   JOURNAL = {Adv. Math.},
  FJOURNAL = {Advances in Mathematics},
    VOLUME = {113},
      YEAR = {1995},
    NUMBER = {1},
     PAGES = {45--117},
      ISSN = {0001-8708,1090-2082},
   MRCLASS = {13C14 (13A35 13B99 13H10)},
  MRNUMBER = {1332808},
MRREVIEWER = {Ian\ M.\ Aberbach},
       DOI = {10.1006/aima.1995.1035},
}

@article {Hei93,
    AUTHOR = {Heitmann, Raymond C.},
     TITLE = {Characterization of completions of unique factorization
              domains},
   JOURNAL = {Trans. Amer. Math. Soc.},
  FJOURNAL = {Transactions of the American Mathematical Society},
    VOLUME = {337},
      YEAR = {1993},
    NUMBER = {1},
     PAGES = {379--387},
      ISSN = {0002-9947,1088-6850},
   MRCLASS = {13B35 (13C15 13F15)},
  MRNUMBER = {1102888},
MRREVIEWER = {Jos\'e\ A.\ Hermida-Alonso},
       DOI = {10.2307/2154327},
}

@article {Fed83,
    AUTHOR = {Fedder, Richard},
     TITLE = {{$F$}-purity and rational singularity},
   JOURNAL = {Trans. Amer. Math. Soc.},
  FJOURNAL = {Transactions of the American Mathematical Society},
    VOLUME = {278},
      YEAR = {1983},
    NUMBER = {2},
     PAGES = {461--480},
      ISSN = {0002-9947,1088-6850},
   MRCLASS = {13H10 (13D03 14B05)},
  MRNUMBER = {701505},
MRREVIEWER = {D.\ Kirby},
       DOI = {10.2307/1999165},
}

@incollection{FW89,
    AUTHOR = {Fedder, Richard and Watanabe, Keiichi},
     TITLE = {A characterization of {$F$}-regularity in terms of
              {$F$}-purity},
 BOOKTITLE = {Commutative algebra ({B}erkeley, {CA}, 1987)},
    SERIES = {Math. Sci. Res. Inst. Publ.},
    VOLUME = {15},
     PAGES = {227--245},
 PUBLISHER = {Springer, New York},
      YEAR = {1989},
      ISBN = {0-387-96990-X},
   MRCLASS = {13D05 (13F50 13H10)},
  MRNUMBER = {1015520},
       DOI = {10.1007/978-1-4612-3660-3\_11},
}

@article {DT23,
    AUTHOR = {Datta, Rankeya and Tucker, Kevin},
     TITLE = {On some permanence properties of (derived) splinters},
   JOURNAL = {Michigan Math. J.},
  FJOURNAL = {Michigan Mathematical Journal},
    VOLUME = {73},
      YEAR = {2023},
    NUMBER = {2},
     PAGES = {371--400},
      ISSN = {0026-2285,1945-2365},
   MRCLASS = {13A35 (14G45)},
  MRNUMBER = {4584866},
MRREVIEWER = {Geoffrey\ D.\ Dietz},
       DOI = {10.1307/mmj/20205951},

}

\end{document}